\documentclass[11pt]{article}
\usepackage[T1]{fontenc}
\usepackage{lmodern}
\usepackage[tmargin=1in,bmargin=1in,lmargin=1in,rmargin=1in]{geometry}
\usepackage{amsthm}
\usepackage{amsmath}
\usepackage{amssymb}
\usepackage{bm}
\usepackage{mathrsfs}
\usepackage[dvips]{graphicx}

\usepackage{algorithmic}
\usepackage{algorithm}

\usepackage{color}

\usepackage{url}

\usepackage{supertabular}

\newtheorem{thm}{Theorem}
\newtheorem{cor}[thm]{Corollary}
\newtheorem{lemma}[thm]{Lemma}

\theoremstyle{definition}

\theoremstyle{remark}

\numberwithin{thm}{section}

\DeclareMathAlphabet{\mathsfsl}{OT1}{cmss}{m}{sl}

\renewcommand{\phi}{\varphi}

\newcommand{\grad}{\nabla}

\newcommand{\argmax}{\operatorname*{arg\; max}}

\newcommand{\Expect}{\operatorname{\mathbb{E}}}

\newcommand{\range}{\operatorname{range}}

\newcommand{\bx}{\boldsymbol{x}}

\newcommand{\by}{\boldsymbol{y}}

\newcommand{\bz}{\boldsymbol{z}}

\newcommand{\bP}{\boldsymbol{P}}

\def\bbC{\mathbb{C}}

\def\reals{\mathbb{R}}
\def\bx{\boldsymbol{x}}
\def\be{\boldsymbol{e}}

\def\b0{\mathbf{0}}
\def\bP{\boldsymbol{P}}

\def\bA{\boldsymbol{A}}

\def\ba{\boldsymbol{a}}
\def\bw{\boldsymbol{w}}
\def\bd{\boldsymbol{d}}

\def\bI{\mathbf{I}}

\def\im{\mathrm{im}}
\def\tr{\mathrm{tr}}

\def\Sp{\mathrm{Sp}}

\usepackage{algorithmic}
\usepackage{algorithm}
\newcommand{\di}{{\,\mathrm{d}}}

\usepackage[utf8]{inputenc} % allow utf-8 input
\usepackage[T1]{fontenc}    % use 8-bit T1 fonts
\usepackage{hyperref}       % hyperlinks
\usepackage{url}            % simple URL typesetting
\usepackage{booktabs}       % professional-quality tables
\usepackage{amsfonts}       % blackboard math symbols
\usepackage{nicefrac}       % compact symbols for 1/2, etc.
\usepackage{microtype}      % microtypography

\title{Phase retrieval using alternating minimization in a batch setting}

\author{
Teng Zhang\\%\thanks{Use footnote for providing further
  Department of Mathematics\\
 University of Central Florida\\
  \texttt{teng.zhang@ucf.edu} \\
}
\title{Phase retrieval using alternating minimization in a batch setting}

\begin{document}

\maketitle
\begin{abstract}
This paper considers the problem of phase retrieval, where the goal is to recover a signal $\bz\in\bbC^n$ from the observations $y_i=|\ba_i^*\bz|$, $i=1,2,\cdots,m$. While many algorithms have been proposed, the alternating minimization algorithm is still one of the most commonly used and the simplest methods. Existing works have proved that when the observation vectors $\{\ba_i\}_{i=1}^m$ are sampled from a complex norm distribution $CN(0,\bI)$, the alternating minimization algorithm recovers the underlying signal with a good initialization when $m=O(n)$, or with random initialization when $m=O(n^2)$, and it is conjectured that random initialization succeeds with $m=O(n)$~\cite{Waldspurger2016}. This work proposes a modified alternating minimization method in a batch setting and proves that when $m=O(n\log^{5}n)$, the proposed algorithm with random initialization recovers the underlying signal with high probability. The proof is based on the observation that after each iteration of alternating minimization, with high probability, the correlation between the direction of the estimated signal and the direction of the underlying signal increases.  %In comparison, many current algorithms with theoretical guarantees require a good initialization, and the best guarantee for the alternating minimization with random initialization only holds when the number of measurements satisfy $m=O(n^2)$. Our empirical results corroborate our theoretical results and show that our method succeeds when.
\end{abstract}

\section{Introduction}
This article concerns the phase retrieval problem as follows: let $\bz\in\bbC^n$ be an unknown vector; given $m$ known sensing vectors $\{\ba_i\}_{i=1}^m\in\bbC^n$ and the observations
\[
y_i=|\ba_i^*\bz|, i=1,2,\cdots,m,
\]
then can we reconstruct $\bz$ from the observations $\{y_i\}_{i=1}^m$? This problem is motivated from the applications in imaging science, and we refer interested readers to ~\cite{Shechtman2015} for more
detailed discussions on the background in engineering. In addition, this problem has applications in other areas of sciences and engineering as well, as discussed in \cite{Candes7029630}.

Because of the practical ubiquity of the phase retrieval problem, many algorithms and theoretical analysis have been developed for this problem. For example, an interesting recent approach is based on convex relaxation~\cite{Chai2011,Candes_PhaseLift,Waldspurger2015}, that replaces the non-convex measurements by convex measurements through relaxation. Since the associated optimization problem is convex, it has interesting properties such as convergence to the global minimizer, and it has been shown that under some assumptions on the sensing vectors, this method recovers the correct $\bz$~\cite{Candes2014,Gross2015}. However, since these algorithms involve semidefinite programming for $n\times n$ positive semidefinite matrices, the computational cost could be prohibitive when $n$ is large. Recently, several works \cite{pmlr-v54-bahmani17a,Goldstein2016,Hand2016,Hand20162} proposed and analyzed an alternate convex method that uses linear programming instead of  semidefinite programming, which is more computationally efficient, but the program itself requires an ``anchor vector'', which needs to be a good approximate estimation of $\bz$.

Another line of works are based on Wirtinger
flows, i.e., gradient flow in the complex setting~\cite{Candes7029630,NIPS2015_5743,Zhang:2016:PNP:3045390.3045499,NIPS2016_6319,cai2016,NIPS2016_6061,Soltanolkotabi2017}. Some theoretical justifications are also provided \cite{Candes7029630,Soltanolkotabi2017}. However, since the objective functions are nonconvex, these algorithms require careful initializations, which are usually only justified when the measurement vectors follow a very specific model, for example, when the observation vectors $\{\ba_i\}_{i=1}^m$ are sampled from a complex normal distribution $CN(0,\bI)$. That is, both its real component and its imaginary component follows from a real Gaussian distribution of $N(0,\bI/2)$. In addition, there are technical issues in implementation such as choosing step sizes, which makes the implementation slightly more complicated.

To cope with the nonconvexity of the phase retrieval problem, Sun et al. \cite{7541725} tries to understand the geometric landscape of a nonconvex objective function associated with phase retrieval, and  proved that when $m=O(n\log^3n)$, their cost function has no bad critical point, and as a result, arbitrary initialization is sufficient and a trust-region method (TRM) can be applied to obtain the solution. However, this method is more complicated than the alternate minimization algorithm as described below, due to its specific objective function and the associated trust-region method.

The most widely used method is perhaps the alternate minimization algorithm and its variants~\cite{Gerchberg72,Fienup78,Fienup82}, that is based on alternating projections onto nonconvex sets~\cite{Bauschke03}.  This method is very simple to implement and is parameter-free. However, since it is a nonconvex algorithm, its properties such as convergence are only partially known. Netrapalli et al. \cite{Netrapalli7130654} studied a resampling version of this algorithm and established its convergence as the number of measurements $m$ goes to infinity when the measurement vectors are independent standard complex normal vectors. Marchesini et al. \cite{Marchesini2016815} studied and demonstrated the necessary and sufficient conditions for the local convergence of this algorithm. Recently, Waldspurger \cite{Waldspurger2016} showed that when $m \geq Cn$ for sufficiently large $C$, the alternating minimization algorithm succeeds with high
probability, provided that the algorithm is carefully initialized. In addition, with random initialization, the algorithm succeeds with $m\geq C n^2$. This work also conjectured that the alternate minimizations algorithm with random initialization succeeds with $m\geq Cn$.

The contribution of this work is to show that a modified version of the alternating minimization algorithm and random initialization succeeds  with high probability when $m=O(n\log^{3}n)$, which partially verifies the conjecture that the alternating minimization algorithm succeeds with high probability when $m=O(n)$. Compared with the previous methods based on Wirtinger flows and linear programming, the proposed algorithm is more practical since it does not require a good initialization, and compared with the existing works that also do not depend on good initializations such as semidefinite programming and \cite{7541725}, the proposed alternating minimization algorithm  is simpler and easier to implement.%We believe that our results in this paper are of interest, and may have implications. In particular, it shows that a schoticatic gradient descent algorithm with random initialization  succeeds with high probability even when local minimizers exist.

The paper is organized as follows. Section~\ref{sec:main} presents the algorithm and the main results of the paper, and the proof of the key component, Theorem~\ref{thm:main1}, is given in Section~\ref{sec:proof}. We run simulations  to verify Theorem~\ref{thm:main1} in Section~\ref{sec:simu}.%???? %The proof of \eqref{eq:concentration2} is more complicated and we refer the reader to the supplementary material.

\section{Algorithm and Main Results}\label{sec:main}
The alternating minimization method is one of the earliest methods that was introduced for phase retrieval problems~\cite{Gerchberg72,Fienup78,Fienup82}, and it is based on alternating projections onto nonconvex sets~\cite{Bauschke03}. Let  $\bA\in\bbC^{m\times n}$ be a matrix with rows given by $\ba_1^*,\ba_2^*,\cdots,\ba_m^*$, the goal of this algorithm is to find a vector in $\bbC^m$ such that it lies in both the set $\mathcal{S}=\range(\bA)\in\bbC^m$ and the set of correct amplitude $\mathcal{A}=\{\bw\in\bbC^m: |\bw_i|=y_i\}$. For this purpose, the algorithm  picks an initial guess in $\bbC^m$, and alternatively projects it to both sets.
The projections $P_\mathcal{S}, P_\mathcal{A}: \bbC^m\rightarrow\bbC^m$ can be defined by
\[
P_\mathcal{S}(\bw)=\bA(\bA^*\bA)^{-1}\bA^*\bw,\,\,\,[P_\mathcal{A}(\bw)]_i=y_i\frac{\bw_i}{|\bw_i|},
\]
and the alternating minimization algorithm is given by
\begin{equation}\label{eq:alternate_minimization}
\bw^{(k+1)}=P_{\mathcal{S}}P_{\mathcal{A}}\bw^{(k)}.
\end{equation}

In fact, the alternating minimization method can be explicitly written down as follows.  Writing $\bw^{(k)}=\bA\bx^{(k)}$ and let $\be_i\in\bbC^m$ be the indicator vector of the $i$-th coordinate, then the update formula is
\[
\bA\bx^{(k+1)}=\bA(\bA^*\bA)^{-1}\bA^*\left(\sum_{i=1}^m|\ba_i^*\bz|\frac{\ba_i^*\bx^{(k)}}{|\ba_i^*\bx^{(k)}|}\be_i\right)
=\bA(\bA^*\bA)^{-1}\left(\sum_{i=1}^m\frac{|\ba_i^*\bz|}{|\ba_i^*\bx^{(k)}|}\ba_i^*\bx^{(k)}\ba_i\right),
\]
which implies
\begin{equation}\label{eq:algorithm}
\bx^{(k+1)}=(\bA^*\bA)^{-1}\left(\sum_{i=1}^m\frac{|\ba_i^*\bz|}{|\ba_i^*\bx^{(k)}|}\ba_i\ba_i^*\bx^{(k)}\right).
\end{equation}

Define
\begin{equation}\label{eq:gx}
g_i(\bx)=\frac{|\ba_i^*\bz|}{|\ba_i^*\bx|}\ba_i\ba_i^*\bx, \,\,\,\,g(\bx)=\sum_{i=1}^m g_i(\bx),\,\,\,T(\bx)=(\bA^*\bA)^{-1}g(\bx)
\end{equation}
then the algorithm \eqref{eq:algorithm} can be written as
\begin{equation}\label{eq:algorithm1}
\bx^{(k+1)}=T(\bx^{(k)}).
\end{equation}

In this work, we will consider the algorithm~\eqref{eq:algorithm} in a batch setting. Similar to AltMinPhase \cite{Netrapalli7130654}, we divide the sampling vectors $\ba_i$ (the rows of the matrix $\bA$) and corresponding observations $y_i$ into $B$ disjoint blocks $(\by^{(1)},\bA^{(1)}), \cdots, (\by^{(B)},\bA^{(B)})$ of roughly equal size, and perform alternating minimization \eqref{eq:alternate_minimization} to the disjoint blocks cyclically. The procedure is summarized as Algorithm~\ref{alg:main}, where $T^{(k)}$ represents the alternating minimization operator with the $k$-th block $(\by^{(k)},\bA^{(k)})$. We remark that while it is similar to AltMinPhase, this algorithm uses partitions cyclically, rather than only using each partition once. As a result, it only requires finite observations to estimate  $\bz$ exactly, which is different than the method in~\cite{Netrapalli7130654}.

\begin{algorithm}%[tb]
\caption{\ Alternating minimization in a batch setting}
\label{alg:main}
{\bf Input:}  The sampling vectors $\bA\in\bbC^{m\times n}$ and corresponding observations $\by\in\bbC^m$ partitioned into $B$ disjoint blocks $(\by^{(1)},\bA^{(1)}), \cdots, (\by^{(B)},\bA^{(B)})$ of roughly equal size. \\
{\bf Output:} An estimator of the underlying signal $\bz$.\\
{\bf Steps:}\\
{\bf 1:} Let $\bx^{(0)}$ be a random unit vector in $\bbC^n$, $k=0$.\\
{\bf 2:} Repeat \\
{\bf 3:} $\bx^{(k+1)}\leftarrow T^{(\mathrm{mod}(k,B)+1)}\bx^{(k)}$, $k=k+1$ \\
{\bf 4:} Until Convergence\\
{\bf Output:} $\lim_{k\rightarrow\infty} \bx^{(k)}$.
\end{algorithm}

\subsection{Main Result}
Before we state our main result, we present an auxiliary function and its related properties as follows. We remark that in the following statements and proofs, we use $c, c'
, C, C'$ to denote any fixed constants as $m, n \rightarrow\infty$. Depending on the context, they might denote different values in different equations and expressions. 
\begin{thm}\label{thm:main}
There exists $C_0,C_0',C_1,C_2,C_3$ that does not depend on $n$ and $m$, such that when $m>C_0C_0' n\log^{5}n$ and $B=C_0\log n$ satisfies $n>C_3\log m$ and $m/B>C_3 n$, then with probability at least $1-C/\log n-\exp(-Cn)-2B/\log^2 n-BC_1\exp(-C_2m/B)$, Algorithm~\ref{alg:main} recovers the underlying $\bz$ multiplication by a global phase in the sense that $\lim_{k\rightarrow\infty}|\bz^*\bx^{(k)}|=1$.
\end{thm}
We remark that when $n$ and $m$ goes to infinity together under the assumption that $\frac{m}{n\log^5n}\rightarrow \infty$ and $\frac{n}{\log m}\rightarrow \infty$, then the conditions in Theorem~\ref{thm:main} are satisfied and the probability in Theorem~\ref{thm:main} goes to $1$.

\subsection{Sketch of the proof}
The proof of the main result, Theorem~\ref{thm:main}, can be divided into three steps. First, the random initialization in step 1 Algorithm~\ref{alg:main} exhibits a slight correlation with the ground truth. Then one may run a batched version of alternating projections by partitioning the
measurements into $O(\log n)$ batches. Since the batches are independent of each other, the second step proves that projecting onto the measurements of each batch will (with high probability) iteratively improve the estimation until it has a constant correlation with the ground truth. Finally, Theorem 3.1 of [25] gives that (with high probability) alternating projections converges to the ground truth provided the seed has a constant correlation with the ground truth. %In particular, this last step allows us to
%reuse the batches instead of drawing fresh randomness.

%In particular, the details of step 1 and step 3 are 

\subsubsection{Step 1: random initialization}
Throughout the paper, we define the $\theta(\bx)$ by $\sin^{-1}(|\bx^*\bz|/\|\bx\|\|\bz\|)$, which can be understood as the ``angle'' between $\bx$ with the hyperplane that is orthogonal $\bz$ (though here the angle is not well defined since $\bx$ and $\bz$ are complex-valued). For example, when $\theta(\bx)=\pi/2$, then there exists a constant $c\in\mathbb{C}$ such that $\bx=c\bz$; when $\theta(\bx)=0$, then $\bx$ is orthogonal to $\bz$ in the sense that $\bx^*\bz=0$.

For Algorithm~\ref{alg:main}, the random initialization has a slight correlation with $\bz$ as follows:
\begin{lemma}\label{lemma:initialization}
For any fixed $\bz\in\mathbb{C}^n$ and random unit vector $\bx_0\in\mathbb{C}^n$, with probability $1-C/\log n-\exp(-Cn)$, $\theta(\bx^{(0)})>\sin^{-1}(\frac{1}{2\log n\sqrt{n}})$.
\end{lemma}
\begin{proof}[Proof of Lemma~\ref{lemma:initialization}]
WLOG assume $\bz=(1,0,\cdots,0)$, then $|\bx^{(0)*}\bz|=|\bx^{(0)}_1|/\|\bx^{(0)}\|$. Using Hanson-Wright inequality~\cite{rudelson2013}  with $\|\bx^{(0)}\|^2=\bx^{(0)*}\bI\bx^{(0)}$, we have that with probability $1-\exp(-Cn)$, $\|\bx^{(0)}\|<2\sqrt{n}$. In addition,  with probability at least $1-C/\log n$, $|\bx^{(0)}_1|>1/\log n$. Combing these two observations, Lemma~\ref{lemma:initialization} is proved. We remark that while~\cite{rudelson2013} presents the  Hanson-Wright inequality for real-valued vectors and matrices, it is straightforward to generalize it to the complex-valued vectors and matrices, by writing any complex number as a pair of real numbers.
\end{proof}

\subsection{Step 2: iterative improvement}
In the second step, we prove that the correlation between $\bx^{(i)}$ and $\bz$ over each iteration improves (with high probability). We first introduce a function $h(\theta): \reals\rightarrow\reals$ and an auxiliary lemma on the property of $h(\theta)$. 
\begin{lemma}\label{lemma:conjecture} Let $a_1$ and $a_2$ be two complex variables independently sampled from a complex normal distribution $CN(0,1)$. Let $h(\theta)=\Expect_{a_1,a_2\sim CN(0,1)} |a_1||a_1\sin\theta+a_2\cos\theta|$, then  there exists $c>0$ such that for all $0<\theta<\pi/2$,
$h'(\theta)\geq c\min(\theta,\pi/2-\theta).$
In addition, there exists $c'>0$ such that $\min_{0\leq \theta<\pi/2}h(\theta)<c'.$
\end{lemma}

For the main result in this step, we investigate $T(\bx)$ as defined in \eqref{eq:algorithm1}, rather than $T^{(k)}$ as defined in Algorithm~\ref{alg:main}. However, $T$ is a random operator that exhibits the same distribution as each $T^{(k)}$. 
\begin{thm}\label{thm:main1}
Assuming that $\{\ba_i\}_{i=1}^m$ are i.i.d. sampled from complex normal distribution $CN(0,1)$, then there exists $C_3,C_4>0$ such that if $m>C_3n$ and $n>C_3\log m$, then for any fixed $\bx\in\bbC^n$, with probability at least $1-2/\log^2n$,
\[
\theta(T(\bx))>\left(1-C_4\frac{n}{\theta(\bx)\sqrt{m}}\right)\left(\theta(\bx)+\tan^{-1} \frac{h'(\theta(\bx))}{h(\theta(\bx))}\right).
\]
\end{thm}
Theorem~\ref{thm:main1} is the key element of this work since it describes the performance of the alternating minimization in each iteration. Its proof is rather technical and it is deferred to Section~\ref{sec:proof}.

\subsection{Step 3: complete the proof}
To complete the proof of Theorem~\ref{thm:main}, we apply the following lemma, which is a result of  \cite[Theorem 3.1]{Waldspurger2016}. Similar to Theorem~\ref{thm:main1}, it is a result for the operator $T$ defined in \eqref{eq:algorithm1}, instead of $T^{(i)}$ as defined in Algorithm~\ref{thm:main1}. %While the original statement measures the distance between $\bx$ and $\bz$ by $\min_{\phi\in\reals}\|e^{i\phi}\bx-\bz\|$, applying the fact that the operator $T(\bx)$ is independent of $\|\bx\|$ and only depends on $\bx/\|\bx\|$, we can use $\theta(\bx)$ to measure the distance between $\bx$ and $\bz$ instead.
\begin{lemma}[Theorem 3.1 in \cite{Waldspurger2016}]\label{lemma:convergence}
Assuming that $\{\ba_i\}_{i=1}^m$ are i.i.d. sampled from complex normal distribution $CN(0,1)$, then there exists $\epsilon, C_1, C_2, M >0$ and $0<\delta<1$ such that if $m\geq M n$, then with probability $1-C_1\exp(-C_2m)$, for all $\bx$ such that \[
\inf_{\phi\in \reals}\|e^{i\phi}\bz-\bx\|\leq \epsilon \|\bz\|,
\] then
\[
\inf_{\phi\in \reals}\|e^{i\phi}\bz-T(\bx)\|\leq \delta\inf_{\phi\in \reals}\|e^{i\phi}\bz-\bx\|.
\]
\end{lemma}

Combining this result with the previous steps, we proved Theorem~\ref{thm:main}.

\begin{proof}[Proof of Theorem~\ref{thm:main}]
Applying Lemma~\ref{lemma:convergence} to the $B$ operators $T^{(i)}$ with $i=1, \cdots, B$, then we have the following result: if $m/B>Mn$, then with probability $1-BC_1\exp(-C_2m)$, for all $\bx$ such that \[
\inf_{\phi\in \reals}\|e^{i\phi}\bz-\bx\|\leq \epsilon \|\bz\|,
\] and for all $1\leq i\leq B$, 
\begin{equation}\label{eq:contr5}
\inf_{\phi\in \reals}\|e^{i\phi}\bz-T^{(i)}(\bx)\|\leq \delta\inf_{\phi\in \reals}\|e^{i\phi}\bz-\bx\|.
\end{equation} 
Then as long as \begin{equation}\label{eq:sufficient}
\inf_{\phi\in \reals}\|e^{i\phi}\bz-\bx^{(i)}\|\leq \epsilon \|\bz\|,\,\,\text{ for some $0\leq i\leq B$}
\end{equation} then the sequence $\inf_{\phi\in \reals}\|e^{i\phi}\bz-\bx^{(k)}\|$ for $k=i, i+1,\cdots$ will converge linearly to zero.

Consider that the operator $T^{(i)}(\bx)$ are invariant to the scale of $\bx$ and $\inf_{\phi\in \reals,c\in\mathbb{C}}\|e^{i\phi}\bz-c\bx^{(i)}\|\leq \theta(\bx^{(i)})\|\bz\|$, the sufficient condition in \eqref{eq:sufficient} can be further reduced to
\begin{equation}\label{eq:sufficient1}
\theta(\bx^{(i)})\leq \epsilon,\,\,\text{ for some $0\leq i\leq B$}
\end{equation}
That is, it is sufficient to prove that \eqref{eq:sufficient1} holds with high probability. If for all $0\leq i< B$, $\theta(\bx^{(i)})<\frac{\pi}{2}-\epsilon$, then Lemma~\ref{lemma:conjecture} implies that there exists $c_{\epsilon}>0$ such that $\tan^{-1}\frac{h'(\theta(\bx^{(i)}))}{h(\theta(\bx^{(i)}))}>c_{\epsilon}\theta(\bx^{(i)})$ for all $0\leq i< B$. Since each batch has $m/B$ observations and for $1\leq i\leq B$, $T^{(i)}$ is independent with $\bx^{(i)}$,  $m/B>C_3n$, and $n>C_3\log m$, Theorem~\ref{thm:main1} implies that for each $0\leq i<B$, with probability $1-2/\log^2n$,
\begin{equation}\label{eq:contr}
\theta(\bx^{(i+1)})>\left[(1+c_{\epsilon})\left(1-C_4\frac{\log n\sqrt{B}}{\theta(\bx^{(i)})\sqrt{m}}\right)\right] \theta(\bx^{(i)}).
\end{equation}
We choose $C_0$ such that $(1+c_\epsilon/2)^{C_0\log n}\sin^{-1}(1/2\log n\sqrt{n})>\pi/2-\epsilon$, and $C_0'$ such that 
\begin{equation}\label{eq:contr1}
(1+c_{\epsilon})\left(1-C_4\frac{1}{\theta(\bx^{(0)})\sqrt{C_0'}}\right)>1+c_\epsilon/2,
\end{equation}
then when  $B=C_0\log n$ and $m=C_0C_0'\log^5n$, applying \eqref{eq:contr}  and induction we can prove that with probability $1-2B/\log^2n$, 
% \[\]

% Applying Lemma~\ref{lemma:initialization} and induction, it can be shown that for each $i$, if $B=\log n$ and $m>C $ with probability $1-2i/\log^2n$,
\begin{equation}\label{eq:contr2}
\theta(\bx^{(B)})>\left[1+\frac{c_{\epsilon}}{2}\right]^B \theta(\bx^{(0)})>\frac{\pi}{2}-\epsilon.
\end{equation}
% Choose $C=\frac{1}{\log(1+c_{\epsilon}/2)}$, then there exists $C_0$ such that for $n>C_0'$, $(1+c_{\epsilon})\left(1-\frac{1}{\log n}\right)>1+c_{\epsilon}/2$, and it can be shown that
% \[
% \left[(1+c_{\epsilon})\left(1-\frac{1}{\log n}\right)\right]^{C\log n} \sin^{-1}\Big(\frac{1}{2\log n\sqrt{n}}\Big)>\frac{\pi}{2},
% \]
then this is a contradiction to the assumption that \eqref{eq:sufficient1} does not hold, i.e., $\theta(\bx^{(B)})$ can not be larger than $\pi/2-\epsilon$. Therefore, there exist  $0\leq i< B$ such that $\theta(\bx^{(i)})>\frac{\pi}{2}-\epsilon$, and Theorem~\ref{thm:main} is proved. The probabilistic estimation in Theorem~\ref{thm:main} comes from the union bound of Lemma~\ref{lemma:initialization}, \eqref{eq:contr5} and \eqref{eq:contr2}.
\end{proof}

%\subsection{Dependence of Parameters...}
\subsection{Discussion}\label{sec:discussion}
Theorem~\ref{thm:main} has several interesting connections with the results within the current literature. First of all, it complements the analysis of AltMinPhase in~\cite{Netrapalli7130654}. While the analysis of AltMinPhase in~\cite{Netrapalli7130654} is one of the first theoretical guarantees for the alternating minimization algorithm, the work has no instruction on how we should divide the samples into distinct blocks, or how we should choose the number of size of blocks. In addition, the analysis requires infinite observations to recover $\bz$ exactly. In comparison, Theorem~\ref{thm:main} gives an estimation of the number of blocks to use. In addition, Theorem~\ref{thm:main1} also shows that when the size of each block is on the order of $O(n)$ up to a logarithmic factor, then each iteration of the algorithm  improves the estimation of $\bz$, in the sense that every iteration decreases the angle between $\bz$ and the estimator.

Our work also partially answers the conjecture from the work \cite{Waldspurger2016} that when the initialization is randomly chosen and $m=O(n)$, the alternating minimization algorithm succeeds with high probability. In comparison, we proved that the alternating minimization algorithm in a batch setting succeeds with $m=O(n\log^{5}n)$, which is an improvement from the estimation $m=O(n^2)$ in \cite{Waldspurger2016} (though we remark that the result in \cite{Waldspurger2016} is for the non-batch setting).

An interesting observation from \cite{Waldspurger2016} is the existence of stationary points when $m<O(n^2)$. In comparison, Theorem \ref{thm:main} shows that the algorithm avoids these stationary points from random initialization. In this sense, Theorem \ref{thm:main} is very different from most existing theoretical guarantees for phase retrieval, which are based on the observations that there is no stationary point (or there is no stationary point within a neighborhood of $\bz$).

We also emphasize the result in this work can be applied to settings other than $\ba_i\sim CN(0,\bI)$. In fact, most existing works on algorithms that succeed with $m=O(n)$ requires a good initialization, which is constructed under the setting $\ba_i\sim CN(0,\bI)$. For example, \cite{Netrapalli7130654} uses the top eigenvector of $\sum_{i=1}^m |\ba_i^*\bz|^2\ba_i\ba_i^*$, and \cite{NIPS2015_5743} applies a similar estimator with a thresholding-based scheme by using the top eigenvector of
\[
\sum_{i=1}^m |\ba_i^*\bz|^2\ba_i\ba_i^*\mathrm{1}_{|\ba_i^*\bz|^2\leq \frac{9}{m}\sum_{j=1}^{m}|\ba_i^*\bz|^2},
\]
and a similar scheme is also used in \cite{cai2016}. The only exception that we are aware of is  \cite{NIPS2016_6061}, which introduces an orthogonality-promoting initialization that is obtained with a few simple power iterations and the initialization works when the distribution of $\ba_i$ is heavy-tailed. In comparison, random initialization is a much simpler procedure and can be used in the setting that $\{\ba_i\}_{i=1}^m$ are i.i.d. sampled from the complex normal distribution $CN(0,\Sigma)$ in Corollary \ref{cor:generalization} as follows, which suggests that Theorem~\ref{thm:main} still holds under the setting $\ba_i\sim CN(0,\Sigma)$.

\begin{cor}\label{cor:generalization}
Assuming that $\ba_i\sim CN(0,\Sigma)$, $\frac{\|\Sigma\bz\|}{\|\Sigma^{\frac{1}{2}}\bz\|\sqrt{\tr(\Sigma)}}>\frac{c'}{\sqrt{n}}$,  $\tr(\Sigma)\geq \|\Sigma\|_F\log n$, and $\Sigma$ is nonsingular, then Algorithm~\ref{alg:main} converges to the underlying $\bz$ under the assumptions stated in Theorem~\ref{thm:main}.
\end{cor}
\begin{proof}
The proof is based on the observation that it is equivalent to the setting where $\ba_i\sim CN(0,\bI)$. If we let $\tilde{\ba}_i=\Sigma^{-\frac{1}{2}}\ba_i$, $\tilde{\bx}^{(k)}=\Sigma^{\frac{1}{2}}\bx^{(k)}$, and  $\tilde{\bz}=\Sigma^{\frac{1}{2}}\bz$, then the update formula \eqref{eq:algorithm} is equivalent to the setting of estimating $\tilde{\bz}$ with sensing vectors $\{\tilde{\ba}_i\}_{i=1}^m$, with initialization $\tilde{\bx}^{(0)}=\Sigma^{\frac{1}{2}}\bx^{(0)}$ sampled from $CN(0,\Sigma)$.

Now let us investigate the angle between $\tilde{\bx}^{(0)}$ and $\tilde{\bz}^{(0)}$:
\begin{equation}\label{eq:initial0}
\frac{|\tilde{\bx}^{(0)*}\tilde{\bz}|}{\|\tilde{\bx}^{(0)}\|\|\tilde{\bz}\|}
=\frac{|{\bx}^{(0)*}\Sigma {\bz}|}{\|\Sigma^{\frac{1}{2}}{\bx}^{(0)}\|\|\Sigma^{\frac{1}{2}}\bz\|}.
\end{equation}
WLOG we may assume that all elements of $\bx^{(0)}$ are i.i.d. sampled from the complex normal distribution $CN(0,1)$. Then ${\bx}^{(0)*}\Sigma {\bz}$ is distributed according to $CN(0,\|\Sigma\bz\|^2)$, and $|{\bx}^{(0)*}\Sigma {\bz}|>\|\Sigma\bz\|/\log n$ with probability $1-C/\log n$. In addition, Hanson-Wright inequality implies that \[\Pr\{|\|\Sigma^{\frac{1}{2}}{\bx}^{(0)}\|^2-\tr(\Sigma)|>t\}\leq 2\exp(-c\min(\frac{t^2}{\|\Sigma\|_F^2},\frac{t}{\|\Sigma\|})).\]
Since $\tr(\Sigma)\geq \|\Sigma\|_F\log n\geq \|\Sigma\| \log n$, with probability at least $1-2\exp(-c\log n)$, $|\|\Sigma^{\frac{1}{2}}{\bx}^{(0)}\|^2-\tr(\Sigma)|\leq c\tr(\Sigma)$. As a result, the RHS of \eqref{eq:initial0} is larger than
\[
\frac{\|\Sigma\bz\|}{\log n\|\Sigma^{\frac{1}{2}}\bz\|\sqrt{\tr(\Sigma)}}.
\]
If $\frac{\|\Sigma\bz\|}{\|\Sigma^{\frac{1}{2}}\bz\|\sqrt{\tr(\Sigma)}}>\frac{c'}{\sqrt{n}}$, this recovers Lemma~\ref{lemma:initialization}. Following the proof of Theorem~\ref{thm:main}, $\tilde{\bx}^{(n)}$ converges to $\tilde{\bz}$. Since $\Sigma$ is nonsingular, Corollary~\ref{cor:generalization} is proved.
\end{proof}

At last, we emphasize that Theorem~\ref{thm:main} does not apply to the standard alternating minimization algorithm (i.e., not in a batch setting). The reason is that the probabilistic estimation in Theorem~\ref{thm:main1} only holds for a fixed $\bx$ that is independent of $\bA$. However, in the standard alternating minimization algorithm,  $\bx^{(k)}$ for $k>1$ depends on $\bA$, and Theorem~\ref{thm:main1} cannot be used to estimate $\theta(\bx^{(k+1)})$. In comparison, Theorem 3.1 in \cite{Waldspurger2016} applies for all $\bx$ as long as $\bx^{(k)}$ is sufficiently close to $\bz$. It is unclear how we can find a method generalizing Theorem~\ref{thm:main} to the standard alternating minimization algorithm, by ``decoupling'' the dependence of $\bx^{(k)}$ and $\bA$. This is an open question and we consider it as an interesting future direction.

\section{Proof of Theorem~\ref{thm:main1}}\label{sec:proof}
To prove Theorem~\ref{thm:main1}, we first present Lemma~\ref{lemma:expectation}, which gives the exact formula for the expectation of $g_i(\bx)$ for $g_i$ defined in \eqref{eq:gx}. We also present Lemma~\ref{lemma:mean}, which shows that the expectation of $T(\bx)$ is a scalar multiplication of the expectation of $g_i(\bx)$, and Lemma~\ref{lemma:var}, which shows that $T(\bx)$ has a small variance. Combining these three results together, we proved Theorem~\ref{thm:main1}. These lemmas apply the probabilistic setting of Theorem~\ref{thm:main1} by assuming that $\{\ba_i\}_{i=1}^m\sim CN(0,1)$ and $\bx$ is fixed. In the proof, we assume WLOG that $\|\bx\|=\|\bz\|=1$. 
\begin{lemma}\label{lemma:expectation}
 Let $\eta\in[0,2\pi]$ and $\bw$ be chosen such that $\|\bw\|=1$, $\bw\perp\bz$ (i.e., $\bw^*\bz=0$), and $\bx=\sin(\theta)\bz\exp(i\eta)+\cos(\theta)\bw$. Then for $g_i$ defined in \eqref{eq:gx},
\[
\Expect g_i(\bx)= h(\theta) \bx + h'(\theta)\bd,
\]
where $\bd=\cos(\theta)\bz\exp(i\eta)-\sin(\theta)\bw$.
\end{lemma}
\begin{lemma}\label{lemma:mean}
For any $1\leq i\leq m$ and $\Sigma_i=\sum_{1\leq j\leq m, j\neq i}\ba_j\ba_j^*$, 
\begin{equation}\label{eq:mean1}
\left\|\Expect T(\bx)-m\Expect\left(\frac{1}{1+ \tr(\Sigma_i^{-1})}\Sigma_i^{-1}\right)\Expect g_i(\bx)\right\|<Cn/m\end{equation}
\begin{equation}\label{eq:mean2}
\left\|\bz^*\Expect T(\bx)-m\bz^*\Expect\left(\frac{1}{1+ \tr(\Sigma_i^{-1})}\Sigma_i^{-1}\right)\Expect g_i(\bx)\right\|<Cn\sqrt{n}/m
\end{equation}
\end{lemma}
\begin{lemma}\label{lemma:estimation_g}
For $g(x)$ defined in \eqref{eq:gx}, there exists $C>0$ such that
\[
\Pr(\|g(\bx)\|>Ctm)<\exp(-t^2).
\]
\end{lemma}
\begin{lemma}\label{lemma:var}
There exists $C>0$ such that for all $1\leq i\leq n$,
\[
\Expect[\|T(\bx)-\Expect T(\bx)\|^2]<Cn/m,\,\, \text{and}\,\,\mathrm{Var} [\bz^* T(\bx)] < C/m.
\]
\end{lemma}
We first prove Theorem~\ref{thm:main1}, with the proofs of lemmas deferred.
\begin{proof}[Proof of Theorem~\ref{thm:main1}]
Applying the Chebyshev's inequality to Lemma~\ref{lemma:var}, we have that with probability at least $1-2/\log n^2$, we have
\begin{equation}\label{eq:Chebyshev}
\|T(\bx)-\Expect T(\bx)\|<C\sqrt{n}\log  n/\sqrt{m}, \|\bz^*T(\bx)-\bz^* \Expect T(\bx)\|<C\log  n/\sqrt{m}.
\end{equation}
In addition, $\Expect\left(\frac{1}{1+ \tr(\Sigma_i^{-1})}\Sigma_i^{-1}\right)$ is a scalar matrix and \eqref{eq:gassian_sigma} implies that with probability $1/2$, the largest singular value and the smallest singular value of $\Sigma_i$ are both in the order of $1/m$, so there exists some $c=O(1)$ such that its diagonal entries are larger than $c/m$.

Lemma~\ref{lemma:expectation} implies that angle between $\bz^*$ and $g_i(\bx)$ satisfies
 \[
\frac{|\bz^*\Expect g_i(\bx)|}{\|\Expect g_i(\bx)\|}=\sin\left(\theta(\bx)+\tan^{-1}\frac{h'(\theta(\bx))}{h(\theta(\bx))}\right).
 \]
Combining it with  $\|\Expect g_i(\bx)\|\geq 1$ (which follows from Lemma~\ref{lemma:expectation}), $\Expect\left(\frac{1}{1+ \tr(\Sigma_i^{-1})}\Sigma_i^{-1}\right)=\frac{c}{m}\bI$ with $c=o(1)$, \eqref{eq:Chebyshev}, and Lemma~\ref{lemma:mean},
 \begin{align*}
\frac{|\bz^*T(\bx)|}{\|T(\bx)\|}\geq \frac{c\sin\left(\theta(\bx)+\tan^{-1}\frac{h'(\theta(\bx))}{h(\theta(\bx))}\right)-C\frac{\log n}{\sqrt{m}}}{c+C\frac{\sqrt{n}\log n}{\sqrt{m}}}.
 \end{align*}

Then Theorem~\ref{thm:main1} is proved by applying $ \theta(T(\bx))=\sin^{-1}\left(\frac{|\bz^*T(\bx)|}{\|T(\bx)\|}\right)$.
\end{proof}
\subsection{Proof of Auxiliary Lemmas for Theorem~\ref{thm:main1}}
\begin{proof}[Proof of Lemma~\ref{lemma:expectation}]
The proof is based on the observation that $g_i(\bx)$ is the derivative of $|\ba_i^*\bx||\ba_i^*\bz|$. In particular, this work defines the derivatives of real valued functions over complex variables as follows: $\grad f(x)$ is chosen such that
\[
f(x+\Delta x)=f(x)+\mathrm{re}(\grad f(x)^*\Delta x)+o(|\Delta x|).
\]
Then we can define $G(\bx)=\sum_{i=1}^nG_i(\bx)$ with $G_i(\bx)=|\ba_i^*\bx|$. Then we have $g_i(\bx)=\grad G_i(\bx)$ and $g(\bx)=\grad G(\bx)$.

In addition, we can calculate $\Expect G_i(\bx)$. Since the expectation is invariant to unitary transformations of $\bx$ and $\bz$ and $\theta(\bx)=\sin^{-1}(\frac{|\bx^*\bz|}{\|\bx\|\|\bz\|})$, WLOG we may phase $\bz$ so that $\bx^*\bz$ is nonnegative  and assume that $\bz=(1,0,\cdots,0)$ and $\bx=(\sin(\theta),\cos(\theta),0,\cdots,0)$. Then it is clear that \begin{align*}&\Expect[G_i(\bx)]
=\Expect_{[a_1,a_2]\sim CN(0,\bI)}\Big[|a_1||a_1\sin\theta+a_2\cos\theta|\Big]
=h(\theta).
\end{align*}
Since $\Expect[G_i(\bx)]$ only depends on the $\theta(\bx)$  and $\|\bx\|$, its derivative is only nonzero at two directions: $\bx$ and the direction where $\theta(\bx)$ changes most.
% $\by^*\Expect[g_i(\bx)]=\by^*\grad \Expect[G_i(\bx)]=0$ if $\real(\by^*\bx)=\real(\by^*\bd)=0$.
Since the function $G_i$ has the property $G_i(\bx+t\bx)=(1+t)G_i(\bx)$, we have
\[
\bx^*\grad \Expect[G_i(\bx)] = \Expect[G_i(\bx)].
\]
By definition, $\bd$ is the direction where $\theta(\bx)$ changes most, that is, $\bd=\argmax_{\|\by\|=1,\by\in\mathbb{C}^n}\frac{\theta(\bx+t\by)-\theta(\bx)}{t}$, and  $\theta(\bx+t\bd)=\theta(\bx)+t+O(t^2)$. Combining it with $\|\bx+t\bd\|=\|\bx\|+O(t^2)$, we have
\[
\bd^*\grad \Expect[G_i(\bx)]=h'(\theta)_{\theta=\theta(\bx)}.
\]
Combining the above observations together, Lemma~\ref{lemma:expectation} is proved.
\end{proof}

\begin{proof}[Proof of Lemma~\ref{lemma:mean}]
The proof of Lemma~\ref{lemma:mean} is based on an upper bound of $\|\Sigma^{-1}\|$ for $\Sigma=\bA^*\bA$. To start, we apply the result from~\cite[Theorem 1.1]{Tao2010} that for any for any $n\times n$ complex normal matrix $\bA$,
\begin{equation}\label{eq:inverse_norm0}
\Pr\left(\sigma_{\min}(\bA)\leq t\sqrt{n}\right)<t.
\end{equation}
For any $m\times n$ complex normal matrix $\bA$, we denote its smallest singular value by $\sigma_{\min}(\bA)$. Since $\bA$ contains $\lfloor{\frac{m}{n}}\rfloor$ independent submatrices of size $n\times n$, and $\sigma_{\min}(\bA)$ is larger than the smallest singular value of any submatrix of $\bA$, we have
\begin{equation}\label{eq:inverse_norm0}
\Pr\left(\sigma_{\min}(\bA)\leq t\sqrt{n}\right)<t^{\lfloor{\frac{m}{n}}\rfloor},
\end{equation}
We may also apply the result from~\cite[Theorem II.13]{Szarek:survey} that for any $m\times n$ matrix $\Gamma$ that is i.i.d. sampled from real Gaussian distribution $CN(0,1)$, we have
\[
 \Pr\left(\sqrt{m}+\sqrt{n}+t \geq \sigma_1(\Gamma)\geq \sigma_{n}(\Gamma)\geq \sqrt{m}-\sqrt{n}-t \right)>1-\exp(-t^2/2).
\]
Combining it with $\sigma_{n}(\bA)\geq \sigma_{n}(\im(\bA))$ and $\sigma_{1}(\bA)\leq \sigma_1(\mathrm{re}(\bA))+\sigma_1(\im(\bA))$, \begin{equation}\label{eq:gassian_sigma}
\Pr\left\{\sigma_{n}(\bA)\geq \frac{1}{\sqrt{2}}(\sqrt{m}-\sqrt{n}-t),\,\,\sigma_{1}(\bA)\leq {\sqrt{2}}(\sqrt{m}+\sqrt{n}+t)\right\}>1-2\exp(-t^2/2).
\end{equation}
As a result, we have
\begin{align}\label{eq:inverse_norm}
\Pr\left(\sigma_{\min}(\bA)\leq t\right)\leq \min\left(\frac{t}{\sqrt{n}}^{\lfloor{\frac{m}{n}}\rfloor},2\exp\left(-\frac{(\sqrt{m}-\sqrt{n}-t\sqrt{2})^2}{2}\right)\right)\\
\leq \begin{cases}\frac{t}{\sqrt{n}}^{\lfloor{\frac{m}{n}}\rfloor},\,\,\text{if $t<\exp\left(-n\right)$}\\2\exp\left(-\frac{(\sqrt{m}-\sqrt{n}-t\sqrt{2})^2}{2}\right),\,\,\text{if $t\geq \exp\left(-n\right)$}. \end{cases}
\end{align}

Now let us estimate the upper bound of $\Expect \|\Sigma^{-1}\|^2$. Since $\|\Sigma^{-1}\|=\sigma_{\min}(\bA)^{-2}$, so
\begin{align*}
&\Expect \|\Sigma^{-1}\|^2\leq \frac{8}{m^2}+\int_{t=1/m^2}^{\infty}\Pr(\|\Sigma^{-1}\|^2>t)
\leq \frac{8}{m^2}+\int_{t=8/m^2}^{\infty}\Pr(\sigma_{\min}(\bA)<\frac{1}{\sqrt{t}})\di t\\\leq& \frac{8}{m^2}+\int_{t=8/m^2}^{\exp(n/2)}2\exp\left(-\frac{(\sqrt{m}-\sqrt{n}-\sqrt{2/t})^2}{2}\right)\di t+\int_{t=\exp(n/2)}^{\infty}\frac{1}{\sqrt{tn}}^{\lfloor{\frac{m}{n}}\rfloor}\di t\\
\leq & \frac{8}{m^2}+2\exp(n/2)\exp\left(-\frac{(\sqrt{m}/2-\sqrt{n})^2}{2}\right)+\exp(-n/4)\leq \frac{C}{m^2},
\end{align*}
where the last two steps uses the assumption that $m>Cn$ and $n>C\log m$.

In addition, for any fixed $n\times n$ matrix $\Sigma$, Hanson-Wright inequality~\cite{rudelson2013} 
\[
\Pr\left(|\ba_i^*\Sigma\ba_i-\tr(\Sigma)|>t\right)\leq 2\exp\left[-c\min\left(\frac{t^2}{\|\Sigma\|_F^2},\frac{t}{\|\Sigma\|}\right)\right]
\]
implies
\begin{equation}\label{eq:intermediate}
\Pr\left(|\ba_i^*\Sigma\ba_i-\tr(\Sigma)|>t\sqrt{n}\|\Sigma\|\right)\leq 2\exp\left[-c\min(t^2,t\sqrt{n})\right].
\end{equation}
Applying the Sherman–Morrison formula,
\[
 T(\bx)= \Sigma^{-1}\sum_{i=1}^mg_i(\bx)=\sum_{i=1}^m[\Sigma_i^{-1}-\frac{\Sigma_i^{-1}\ba_i\ba_i^*\Sigma_i^{-1}}{1+\ba_i^*\Sigma_i^{-1}\ba_i}]g_i(\bx)
=\sum_{i=1}^m\frac{1}{1+\ba_i^*\Sigma_i^{-1}\ba_i}\Sigma_i^{-1}g_i(\bx),
\]
we have
\begin{align}\nonumber
&\left\|\Expect T(\bx)- \sum_{i=1}^m \Expect \left(\frac{1}{1+ \tr(\Sigma_i^{-1})} \Sigma_i^{-1}\right)\Expect g_i(\bx)
\right\|=m\left\|\Expect\left[ \left(\frac{1}{1+\ba_i^*\Sigma_i^{-1}\ba_i}-\frac{1}{1+ \tr(\Sigma_i^{-1})}\right)\Sigma_i^{-1}g_i(\bx)\right]\right\|\\
\leq &m\Expect\left\| \left(\frac{1}{1+\ba_i^*\Sigma_i^{-1}\ba_i}-\frac{1}{1+ \tr(\Sigma_i^{-1})}\right)\Sigma_i^{-1}g_i(\bx)\right\|\leq m\Expect\left\| \left({\ba_i^*\Sigma_i^{-1}\ba_i}-{ \tr(\Sigma_i^{-1})}\right)\Sigma_i^{-1}g_i(\bx)\right\|\label{eq:mean_diff},
\end{align}
where the last inequality follows from the fact that
\[
\left|\frac{1}{1+\ba_i^*\Sigma_i^{-1}\ba_i}-\frac{1}{1+ \tr(\Sigma_i^{-1})}\right|\leq \left|{\ba_i^*\Sigma_i^{-1}\ba_i}-{ \tr(\Sigma_i^{-1})}\right|.
\]

Applying \cite[Proposition 2.2(1)]{measure2005}, for any $t>1$, with probability $1-t^2\exp(-t^2+1)$, $|\ba_i^*\bz|<t$ and with probability $1-t^{2n}\exp(-(t^2-1)n)$, $\|\ba_i\|<t\sqrt{n}$, which means that with probability $1-t^{2n}\exp(-(t^2-1)n)-t^2\exp(-t^2+1)$, $\|g_i(\bx)\|<t^2\sqrt{n}$. Combining it with the \eqref{eq:intermediate}, the RHS of \eqref{eq:mean_diff} can be estimated by
\begin{align*}
&m\Expect\left\| \left({\ba_i^*\Sigma_i^{-1}\ba_i}-{ \tr(\Sigma_i^{-1})}\right)\Sigma_i^{-1}g_i(\bx)\right\|
\\=&m\Expect_{\{\ba_j\}_{1\leq j\leq m, j\neq i}}\left[\Expect_{\ba_i}\left\| \left({\ba_i^*\Sigma_i^{-1}\ba_i}-{ \tr(\Sigma_i^{-1})}\right)\Sigma_i^{-1}g_i(\bx)\right\|\right]\\
\leq & C n m\Expect_{\{\ba_j\}_{1\leq j\leq m, j\neq i}}\|\Sigma_i^{-1}\|^2.
\end{align*}

Combining it with \eqref{eq:inverse_norm}, we have \eqref{eq:mean1}:
\[
\Big|\Expect T(\bx)- \sum_{i=1}^m \Expect \left(\frac{1}{1+ \tr(\Sigma_i^{-1})} \Sigma_i^{-1}\right)\Expect g_i(\bx)\Big|<C.
\]
The proof of \eqref{eq:mean2} is similar to the proof of \eqref{eq:mean1}, with the estimation of $\|g_i(\bx)\|$ replaced by $|\bz^*g_i(\bx)|$. For $|\bz^*g_i(\bx)|$ we have
\[
|\bz^*g_i(\bx)|=\left|\bz^*\frac{\ba_i^*\bz}{\ba_i^*\bx}\ba_i\ba_i^*\bz\right|=|\ba_i^*\bz|^2,
\]
and $\ba_i^*\bz\sim CN(0,1)$, so with probability at least $1-2\exp(-t^2/2)$, $|\ba_i^*g_i(\bx)|<t^2$.
\end{proof}

% \begin{lemma}\label{lemma:concentration}
% Under the assumption of Theorem~\ref{thm:main1}, then there exists $C_0''>0$ such that with probability at least $1-C/\log^2 n$,
% \begin{equation}\label{eq:concentration1}
% \frac{\|T(\bx)-\Expect g_i(\bx)\|}{\|\Expect g_i(\bx)\|}<\frac{1}{\log n},
% \end{equation}
% and
% \begin{equation}\label{eq:concentration2}
% \frac{|\bz^*(T(\bx)-\Expect g_i(\bx))|}{|\bz^*\Expect g_i(\bx)|}<\frac{1}{\log n}.
% \end{equation}
% \end{lemma}
\begin{proof}[Proof of Lemma~\ref{lemma:estimation_g}]
%We first make a few assumptions as follows:
 %\begin{align}\label{eq:assumption1}
 %&|\ba_i^*\bx|<t, |\ba_i^*\bz|<t,\,\,\,\text{for all $1\leq i\leq m$} \end{align}
 %These assumptions hold with probability at least $1-2m\exp(-t^2/2)$, which can be verified by using the cumulative distribution density of Gaussian variable that $\Phi(x)\geq 1-\frac{1}{2}\exp(-x^2/2)$~\cite[page 8]{ledoux1991probability}.%, and \eqref{eq:assumption1} follows from~\cite{Szarek:survey}.

Let $\Sp(\bx,\bz)$ be the two-dimensional subspace spanned by $\bx$ and $\bz$, and $\bP_{\Sp(\bx,\bz)^\perp}\in\bbC^{n\times n-2}$ be a projector matrix to the $n-2$-dimensional subspace orthogonal to $\Sp(\bx,\bz)$, then
\[
\bP_{\Sp(\bx,\bz)^\perp}g(\bx)=\sum_{i=1}^m|\ba_i^*\bz|\frac{\ba_i^*\bx}{|\ba_i^*\bx|}\bP_{\Sp(\bx,\bz)^\perp}\ba_i,
\]
where $\bP_{\Sp(\bx,\bz)^\perp}\ba_i\in \bbC^{n-2}$ is i.i.d. sampled from $CN(0,\bI)$ and is independent with respect to $|\ba_i^*\bz|\frac{\ba_i^*\bx}{|\ba_i^*\bx|}$. As a result, $\bP_{\Sp(\bx,\bz)^\perp}g(\bx)\in\bbC^{n-2}$ is a vector whose elements are i.i.d. sampled from $CN(0,\sum_{i=1}^m |\ba_i^*\bz|^2)$.

Applying Hansen-Wright inequality~\cite{rudelson2013}, we have
\[
\Pr(\|\bP_{\Sp(\bx,\bz)^\perp}g(\bx)\|^2>2tn \sum_{i=1}^m |\ba_i^*\bz|^2)<\exp(-Cnt^2),
\]
% and
% \[
% \Pr(\sum_{i=1}^m |\ba_i^*\bz|^2> 2m)<\exp(-Cm).
% \]
% As a result, \begin{equation}\label{eq:concentration11}\Pr(\|\bP_{\Sp(\bx,\bz)^\perp}g(\bx)\|^2>4nm)<1-\exp(-Cm)-\exp(-Cn).\end{equation}
% Next, note that
and %by assumption \eqref{eq:assumption1}, we have
\begin{equation*}
\|\bP_{\Sp(\bx,\bz)}g(\bx)\|\leq \sum_{i=1}^n\|\bP_{\Sp(\bx,\bz)}\ba_i\ba_i^*\bz\|\leq \sum_{i=1}^n\|\bP_{\Sp(\bx,\bz)}\ba_i\|^2.
\end{equation*}
In addition, Berstein's inequality implies that there exists $C>0$ such that
\begin{equation}
\Pr(\sum_{i=1}^n|\ba_i^*\bz|^2>Ct)<\exp(-t^2),\,\,\,\,\,\Pr(\sum_{i=1}^n\|\bP_{\Sp(\bx,\bz)}\ba_i\|^2>Ct)<\exp(-t^2).
\end{equation}

Combining these estimations together with
\[
\|g(\bx)\|\leq \|\bP_{\Sp(\bx,\bz)}g(\bx)\|+\|\bP_{\Sp(\bx,\bz)^\perp}g(\bx)\|,
\]
the lemma is proved.
% Applying Hoeffding's inequality,
% \begin{equation}\label{eq:concentration125}
% \Pr(\|\bP_{\Sp(\bx,\bz)}[g(\bx)-\Expect g(\bx)]\|>\frac{m}{\log m})\leq \exp(-\frac{m}{2\log^4m}).
% \end{equation}
% The combination of \eqref{eq:concentration11} and \eqref{eq:concentration125} implies
% \begin{equation}\label{eq:concentration12}
% \Pr(\|g(\bx)-\Expect g(\bx)\|<\frac{m}{\log m}+2\sqrt{nm})\geq 1-\exp(-Cm)-\exp(-Cn)-\exp(-\frac{m}{2\log^4m}).
% \end{equation}
% Applying \eqref{eq:assumption2}, we have
% \begin{equation}\label{eq:concentration13}
% \|mT(\bx)-g(\bx)\|<\sqrt{\frac{n}{m}}\|g(\bx)\|.
% \end{equation}
% Combining \eqref{eq:concentration12}, \eqref{eq:concentration13}, $m>n\log^{C_0''}n$ and the fact that there exists $c>0$ such that $\min_{\|\bx\|=1}\Expect g_i(\bx)\geq c$, we proved \eqref{eq:concentration1} for some $C_0''>0$.
\end{proof}
\begin{proof}[Proof of Lemma~\ref{lemma:var}]
First, we apply the following Lemma, which is a straightforward generalization of the Tensorization of variance theorem~\cite[Theorem 2.3]{Handel_course} to the complex setting:
\begin{lemma}\label{lemma:tensorization}
For complex random variables $X_1,\cdots,X_n$ and $f: \mathbb{C}^n\rightarrow\mathbb{C}$, we have 
\[
\mathrm{Var}[f(X_1,\cdots,X_n)]\leq \Expect\left[\sum_{i=1}^n\mathrm{Var}_i(f(X_1,\cdots,X_n))\right],
\]
where $\mathrm{Var}_i$ is the variance of $f$ with respect to the variable $X_i$ only,  the remaining variables being kept fixed.
\end{lemma}
\begin{proof}
Applying $\mathrm{Var}[f]=\mathrm{Var}[\mathrm{re}(f)]+\mathrm{Var}[\im(f)]$ and the same argument as in the proof of \cite[Theorem 2.3]{Handel_course} for both the real and the imaginary part, the lemma is proved.
\end{proof}

Applying Lemma~\ref{lemma:tensorization}, denote the variance when $\{\ba_i\}_{i\neq j}$ are fixed by
\[
\mathrm{Var}_j(\bz^* T(\bx)),
\]
then we have
\begin{equation}\label{eq:tensor}
\mathrm{Var}(\bz^* T(\bx))\leq \Expect \sum_{j=1}^m[\mathrm{Var}_j(\bz^* T(\bx))].%,\,\,\, \mathrm{Var}(\be_k^* T(\bx))\leq \Expect \sum_{j=1}^m[\mathrm{Var}_j(\be_k^* T(\bx))].
\end{equation}

Then
\begin{align*}
&\mathrm{Var}_j(\bz^* T(\bx)) \leq [\bz^*\Sigma^{-1}\sum_{i=1}^mg_i(\bx)-\bz^*\Sigma_i^{-1}\sum_{i=1,i\neq j}^mg_i(\bx)]^2\\=&\Big[\bz^*\Sigma^{-1}g_j(\bx)-(1+\ba_j^*\Sigma_j^{-1}\ba_j)^{-1}\bz^*\Sigma_j^{-1}\ba_j\ba_j^*\Sigma_j^{-1}\sum_{i=1,i\neq j}^mg_i(\bx)\Big]^2\\ \leq & 2\Big[\bz^*\Sigma^{-1}\ba_j \frac{|\ba_j^*\bz|}{|\ba_j^*\bx|}\ba_j^*\bx\Big]^2+2\Big[\bz^*\Sigma_j^{-1}\ba_j\ba_j^*\Sigma_j^{-1}\sum_{i=1,i\neq j}^mg_i(\bx)\Big]^2\\
\leq & 4\Big[\bz^*\Sigma^{-1}_j\ba_j \frac{|\ba_j^*\bz|}{|\ba_j^*\bx|}\ba_j^*\bx\Big]^2+4\Big[\frac{1}{1+\ba_j^*\Sigma_j^{-1}\ba_j}\bz^*\Sigma^{-1}_j\ba_j\ba_j^*\Sigma^{-1}_j\ba_j \frac{|\ba_j^*\bz|}{|\ba_j^*\bx|}\ba_j^*\bx\Big]^2+2\Big[\bz^*\Sigma_j^{-1}\ba_j\ba_j^*\Sigma_j^{-1}\sum_{i=1,i\neq j}^mg_i(\bx)\Big]^2\\
\leq &8\Big[\bz^*\Sigma^{-1}_j\ba_j |\ba_j^*\bz|\Big]^2+2\Big[\bz^*\Sigma_j^{-1}\ba_j\ba_j^*\Sigma_j^{-1}\sum_{i=1,i\neq j}^mg_i(\bx)\Big]^2.
\end{align*}
Consider that when $\{\ba_i\}_{1\leq i\leq n, i\neq j}$ are fixed and $\ba_j\sim CN(0,1)$, then $\bz^*\Sigma^{-1}_j\ba_j\sim CN(0,\|\bz^*\Sigma^{-1}_j\|^2)$, $\ba_j^*\bz\sim CN(0,\|\bz\|^2)=CN(0,1)$, and $\ba_j^*\Sigma_j^{-1}\sum_{i=1,i\neq j}^mg_i(\bx)\sim CN(0, \|\Sigma_j^{-1}\sum_{i=1,i\neq j}^mg_i(\bx)\|^2)$, so
\begin{align*}
&\Expect \mathrm{Var}_j(\bz^* T(\bx))\leq C \Expect \left[\|\bz^*\Sigma^{-1}_j\|^2 +\|\bz^*\Sigma^{-1}_j\|^2\|\Sigma_j^{-1}\sum_{i=1,i\neq j}^mg_i(\bx)\|^2\right]\\\leq& C \Expect \left[\|\Sigma^{-1}_j\|^2 +\|\Sigma^{-1}_j\|^4\|\sum_{i=1,i\neq j}^mg_i(\bx)\|^2\right].
\end{align*}
Combining it with the estimation of $\|\Sigma_j^{-1}\|$ in \eqref{eq:inverse_norm} and the estimation of $\|\sum_{1\leq i\leq m, i\neq j}^mg_i(\bx)\|$ in Lemma~\ref{lemma:estimation_g} (note that the estimation of $\|\sum_{1\leq i\leq m, i\neq j}^mg_i(\bx)\|$ is identical to the estimation of $\|g(\bx)\|=\|\sum_{1\leq i\leq m}^mg_i(\bx)\|$), we have
\[
\Expect \mathrm{Var}_j(\bz^*T(\bx))<C/m^2.
\]
Applying \eqref{eq:tensor}, we have  $\mathrm{Var} [\bz^* T(\bx)]<C/m$.

Similarly, we can prove the other inequality by showing that any vector $\be_i$, whose $i$-th element is $1$ and other elements are zero,  $\mathrm{Var} [\be_i^* T(\bx)]<C/m$.
\end{proof}

\section{Simulations}\label{sec:simu}
\begin{figure}
\begin{center}
\includegraphics[width=0.45\textwidth]{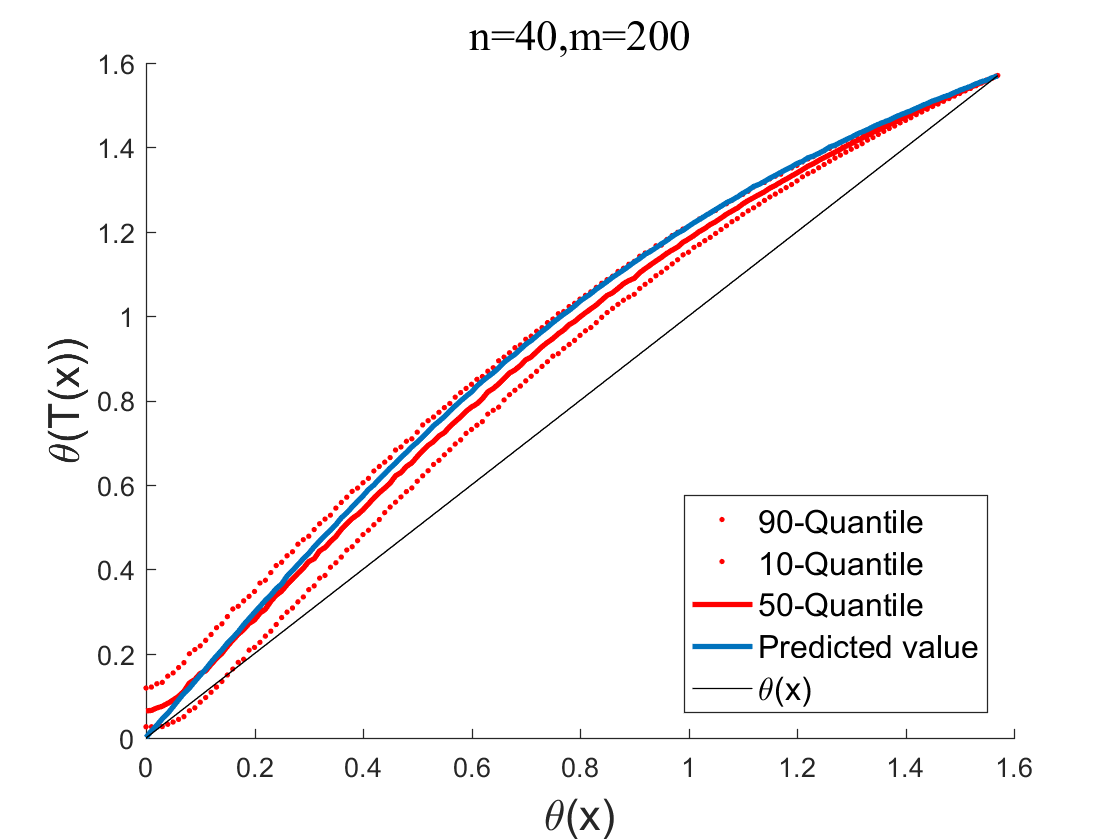}
\includegraphics[width=0.45\textwidth]{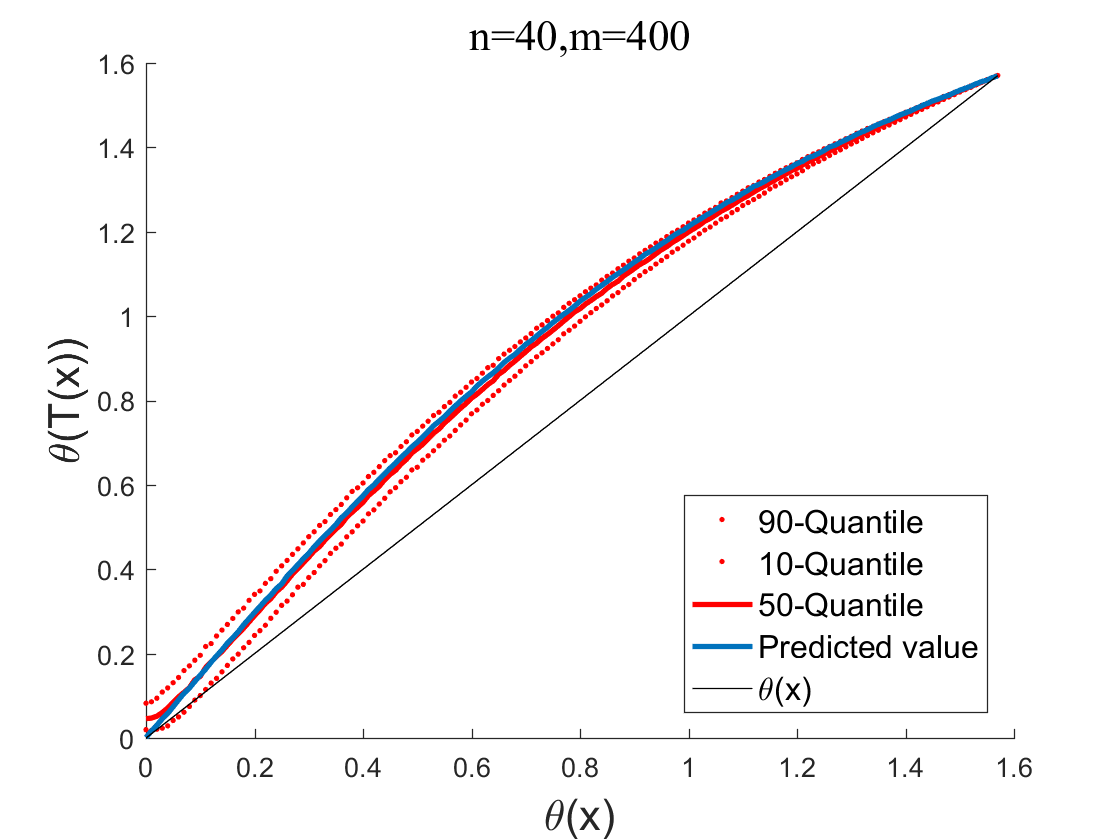}
\includegraphics[width=0.45\textwidth]{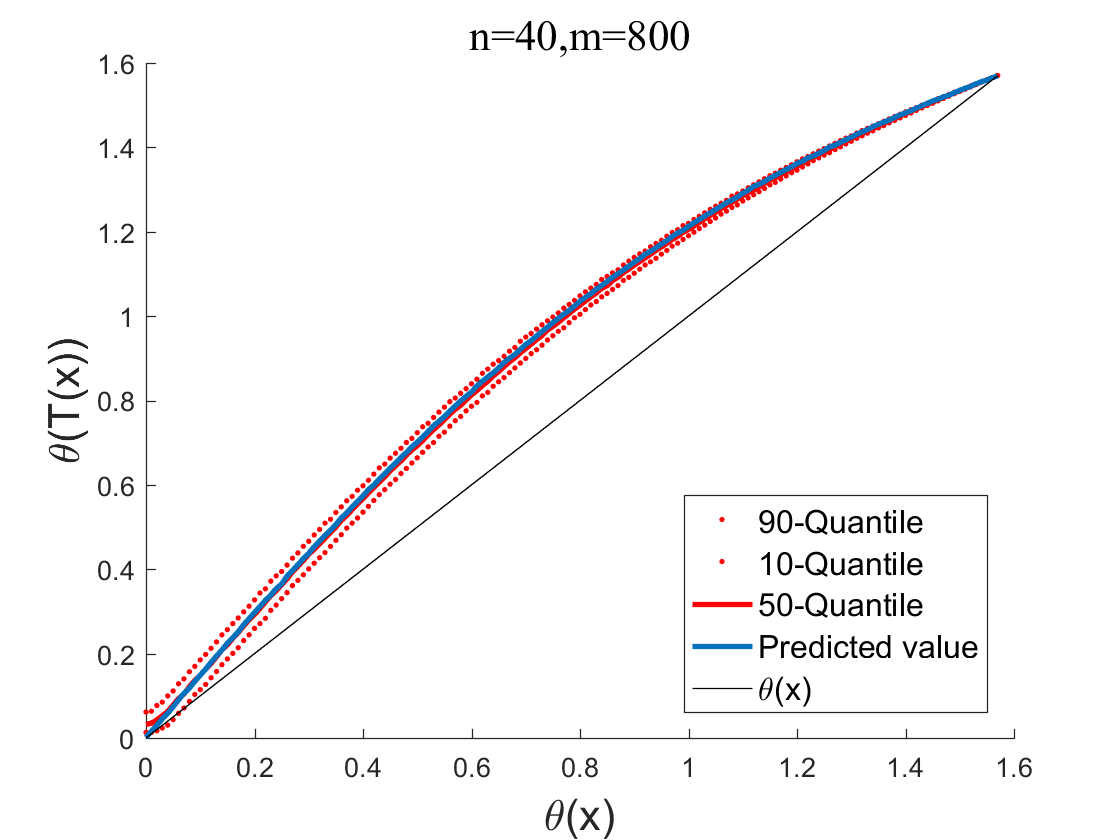}
\includegraphics[width=0.45\textwidth]{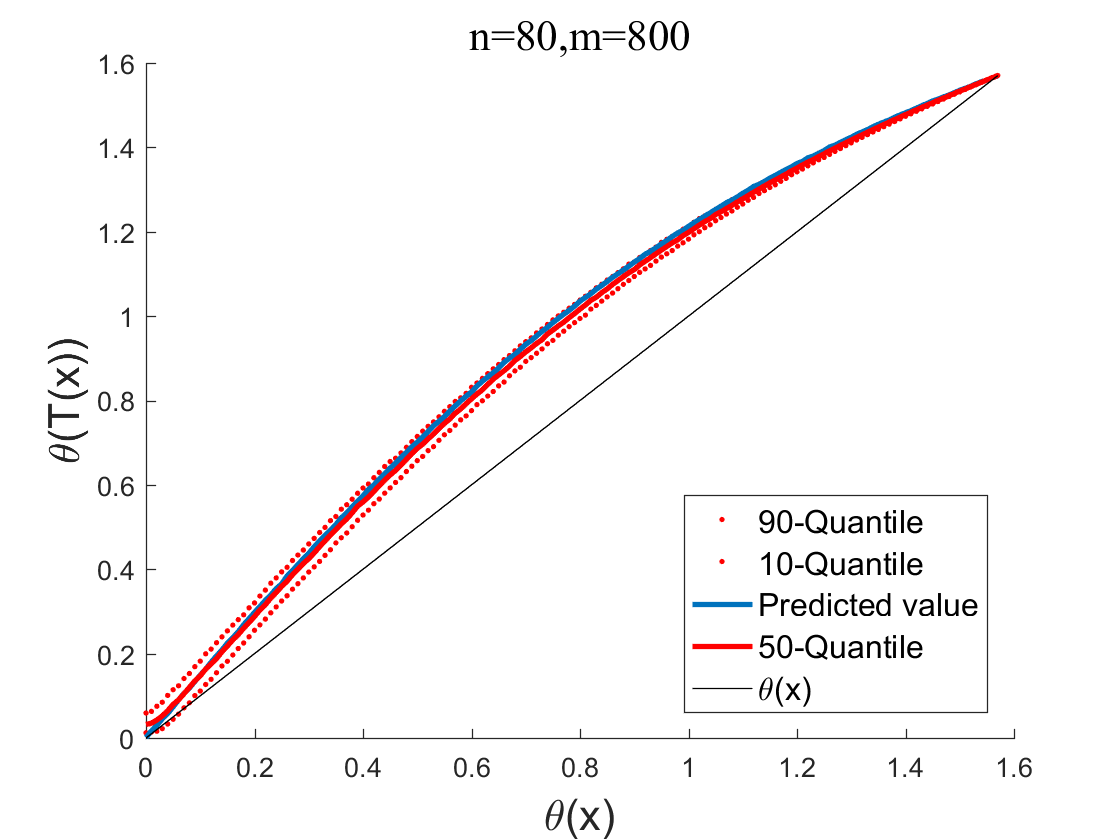}
\end{center}
\caption{Comparison between the predicted and the empirical value of $\theta(T(\bx))$, with various settings of $(n,m)$.}\label{fig:conjecture2}
\end{figure}
This section aims to verify the result in Theorem~\ref{thm:main1}. In particular, we would like to investigate whether empirically, $\theta(\bx)$ and $\theta(T(\bx))$ have the  relation predicted by Theorem~\ref{thm:main1} and its proof:
\begin{equation}\label{eq:predicted}
\theta(T(\bx))\approx \theta(\bx)+\tan^{-1}\frac{h'(\theta(\bx))}{h(\theta(\bx))}.
\end{equation}
For this purpose, we run simulations and compare the empirically observed $\theta(T(\bx))$ and the predicted values.  We run two simulations with different settings of $n, m$. For each setting and each $\theta(\bx)$, we repeat the alternating minimization algorithm randomly by $1000$ times and visualize the $10\%, 50\%, 90\%$ quantile of the observed $\theta(T(\bx))$ in Figure~\ref{fig:conjecture2}, as well as the predicted value in \eqref{eq:predicted}. The figure clearly indicates that our predicted value is close to the empirical values, and as a result, $T(\theta(\bx))>\theta(\bx)$ with high probability as long as $\theta(\bx)$ is not too small, which means that with high probability, the alternating minimization algorithm monotonically reduces the angle between the estimated and the underlying signal. In addition, the variance of the distribution of $\theta(T(\bx))$ is shown to be on the order of $1/\sqrt{m}$.

\section{Summary and Future Directions}
This work analyzes the performance of the alternating minimization algorithm for phase retrieval. Theoretical analysis shows that the angle between the current iteration and the underlying signal is reduced at each iteration with high probability.  Based on this observation, it is shown that alternating minimization in a batch setting with random initialization can recover the underlying signal as long as $m=O(n\log^{5}n)$.% for some $C_0>0$.

A future direction is the analysis of standard alternating minimization without the batch setting. Current work only analyzes the performance of phase retrieval per iteration, as discussed at the end of Section~\ref{sec:discussion}, it does not apply to the standard alternating minimization algorithm. We hope to find a way to uncouple the correlation between $\bx^{(k)}$ and $\bA$, to prove the conjecture that alternating minimization algorithm succeeds with $m=O(n)$. It is also interesting to improve the probabilistic estimation in this work, for example, finding the exact value of $C_0$ and possibly remove the logarithmic factors from the current estimation.

\section{Appendix}
\begin{proof}[Proof of Lemma~\ref{lemma:conjecture}]
Write it in terms of real variables, we have
\[
h(\theta)=\Expect_{a_1,a_2,b_1,b_2\sim N(0,1)}\sqrt{a_1^2+b_1^2}\sqrt{(a_1\sin\theta+a_2\cos\theta)^2+(b_1\sin\theta+b_2\cos\theta)^2}
\]
Using $(\sqrt{f(x)})''=(\frac{1}{2}f(x)^{-1/2}f'(x))'=\frac{1}{2}f(x)^{-1/2}f''(x)-\frac{1}{4}f(x)^{-3/2}f'(x)^2$ and \begin{align*}&[(a_1\sin\theta+a_2\cos\theta)^2+(b_1\sin\theta+b_2\cos\theta)^2]'\\=&2(a_1^2-a_2^2+b_1^2-b_2^2)\sin\theta\cos\theta+2(a_1a_2+b_1b_2)(\cos^2\theta-\sin^2\theta)\\
&[(a_1\sin\theta+a_2\cos\theta)^2+(b_1\sin\theta+b_2\cos\theta)^2]''\\=&2(a_1^2-a_2^2+b_1^2-b_2^2)(\cos^2\theta-\sin^2\theta)-8(a_1a_2+b_1b_2)\cos\theta\sin\theta,\end{align*}
where
\begin{align*}
h''(\theta)=\Expect \frac{\sqrt{a_1^2+b_1^2}}{[f(\theta)]^{\frac{3}{2}}}&\Big[f(\theta)[(a_1^2-a_2^2+b_1^2-b_2^2)(\cos^2\theta-\sin^2\theta)-4(a_1a_2+b_1b_2)\cos\theta\sin\theta]\\-&[(a_1^2-a_2^2+b_1^2-b_2^2)\sin\theta\cos\theta+(a_1a_2+b_1b_2)(\cos^2\theta-\sin^2\theta)]^2\Big],
\end{align*}
and
\[
h''(0)=\Expect \frac{\sqrt{a_1^2+b_1^2}}{[a_2^2+b_2^2]^{\frac{3}{2}}}\Big[[a_2^2+b_2^2][a_1^2+b_1^2-a_2^2-b_2^2]-[a_1a_2+b_1b_2]^2\Big].
\]
Using the fact that when $a_1^2+b_1^2$ and $a_2^2+b_2^2$ are fixed, then under this conditional distribution, $\Expect [a_1a_2+b_1b_2]^2=\frac{1}{2}[a_1^2+b_1^2][a_2^2+b_2^2]$, we have
\begin{align*}
&h''(0)=\Expect \frac{\sqrt{a_1^2+b_1^2}}{[a_2^2+b_2^2]^{\frac{3}{2}}}\Big[\frac{1}{2}[a_2^2+b_2^2][a_1^2+b_1^2]-[a_2^2+b_2^2]^2\Big]
\\=&\Expect \frac{1}{2}[a_2^2+b_2^2]^{-\frac{1}{2}}[a_1^2+b_1^2]^{\frac{3}{2}}-\sqrt{a_1^2+b_1^2}\sqrt{a_2^2+b_2^2}.
\end{align*}
Applying
\[
\Expect (a_1^2+b_1^2)^k=\frac{1}{\pi}\int_{x,y}(x^2+y^2)^ke^{-{x^2-y^2}}\di x\di y=2 \int_{r=0}^\infty r^{2k+1}e^{-r^2}\di r=\int_{z=0}^\infty z^k e^{-z}\di z=\Gamma(k+1),
\]
$h''(0)=\frac{1}{2}\Gamma(\frac{1}{2})\Gamma(\frac{5}{2})-\Gamma(\frac{3}{2})^2=\frac{\pi}{8}>0$.
Using the fact that
\[
h''(\phi)=\frac{\di}{\di\theta}\Expect
\sqrt{(-a_1\sin\phi+a_2\cos\phi)^2+(-b_1\sin\phi+b_2\cos\phi)^2}\sqrt{(a_1\sin\theta+a_2\cos\theta)^2+(b_1\sin\theta+b_2\cos\theta)^2}\Big|_{\theta=0}
\]
and applying the same procedure as in the calculation of $h''(0)$, we have
\begin{equation}\label{eq:second_derivative}
h''(\phi)=\Expect \frac{\sqrt{(-a_1\sin\phi+a_2\cos\phi)^2+(-b_1\sin\phi+b_2\cos\phi)^2}}{[a_2^2+b_2^2]^{\frac{3}{2}}}\Big[[a_2^2+b_2^2][a_1^2+b_1^2-a_2^2-b_2^2]-[a_1a_2+b_1b_2]^2\Big],
\end{equation}
and as a special case,
\begin{align*}
&h''(\frac{\pi}{2})=\Expect \Big[\frac{1}{2}[a_1^2+b_1^2]-[a_2^2+b_2^2]\Big]=-\frac{1}{2}\Gamma(2)=-1.
\end{align*}

Next, we will show that $h''(\theta)$ is well-defined and Lipschitz continuous. In fact, applying \eqref{eq:second_derivative} and the fact that $(-a_1\sin\phi_1+a_2\cos\phi_1)^2-(-a_1\sin\phi_2+a_2\cos\phi_2)^2<|\phi_1-\phi_2|^2(a_1^2+a_2^2)$,
\begin{align*}
|h''(\phi_1)-h''(\phi_2)|
\leq &\Expect|\phi_1-\phi_2|\frac{\sqrt{a_1^2+b_1^2}+\sqrt{a_2^2+b_2^2}}{[a_2^2+b_2^2]^{\frac{3}{2}}}\Big[[a_2^2+b_2^2][a_1^2+b_1^2]+[a_2^2+b_2^2]^2+[a_1a_2+b_1b_2]^2\Big]\\
\leq &\Expect|\phi_1-\phi_2|\frac{\sqrt{a_1^2+b_1^2}+\sqrt{a_2^2+b_2^2}}{[a_2^2+b_2^2]^{\frac{3}{2}}}\Big[\frac{3}{2}[a_2^2+b_2^2][a_1^2+b_1^2]+[a_2^2+b_2^2]^2\Big].
\end{align*}
Then we obtain the Lipschitz continuity of $h''(\theta)$ with Lipschitz factor given by
\[
L=\Expect\frac{\sqrt{a_1^2+b_1^2}+\sqrt{a_2^2+b_2^2}}{[a_2^2+b_2^2]^{\frac{3}{2}}}\Big[\frac{3}{2}[a_2^2+b_2^2][a_1^2+b_1^2]+[a_2^2+b_2^2]^2\Big]=\frac{3}{2}\Gamma(\frac{1}{2})\Gamma(\frac{5}{2})+\Gamma(2).
\]

Then to prove for all $0<\theta<\pi/2$, $h'(\theta)\geq c\min(\theta,\pi/2-\theta)$,
 it is sufficient to verify that \begin{equation}\label{eq:verify}\text{$\min_{\frac{\pi}{16L}<\theta<\frac{\pi}{2}-\frac{1}{2L}}h'(\theta)>c$ for some $c>0$.}\end{equation} Since $h''(\pi/2)=-1$ and $h''(\theta)$ has a Lipschitz factor $L$, $h'(\theta)$ is also Lipschitz continuous with a Lipschitz factor $1 + \pi L$. Therefore, \eqref{eq:verify} can be verified by numerically by checking a few values of $h'(\theta)$ in the interval $\frac{\pi}{16L}<\theta<\frac{\pi}{2}-\frac{1}{2L}$. More specifically, it is sufficient to verify that for $\theta=\frac{\pi}{16L},\frac{\pi}{16L}+\delta,\frac{\pi}{16L}+2\delta,\cdots,\frac{\pi}{2}-\frac{1}{2L}$, $h'(\theta)>c+\delta/(1 + \pi L)$. Using a computer with $\delta=1/10$, it is verified as shown in Figure~\ref{fig:conjecture}.

Based on the Lipschitz continuity of $h'(\theta)$ we can verify the Lipschitz continuity of $h$ in $[0,\pi/2]$. Using a similar procedure as above, we can show that there exists $c'>0$ such that $\min_{0\leq \theta<\pi/2}h(\theta)<c'$, by checking a few functional values of $h(\theta)$ for $\theta\in[0,\pi/2]$.
\end{proof}
To visualize Lemma~\ref{lemma:conjecture}, we randomly reproduce $10^6$ samples of $(a_1,a_2)$, calculate the average values of $h(\theta)$ and $h'(\theta)$ and plot them in Figure~\ref{fig:conjecture}. The right figure verifies that Lemma~\ref{lemma:conjecture} holds. We remark that if $a_1$ and $a_2$ are sampled from real Gaussian distribution $N(0,1)$, then $h(\theta)={\frac{1}{\pi}} [2\theta\sin\theta+2\cos\theta]$, but in the complex setting, the calculation is more complicated and there is no known explicit formula.
\begin{figure}
\begin{center}
\includegraphics[width=0.45\textwidth]{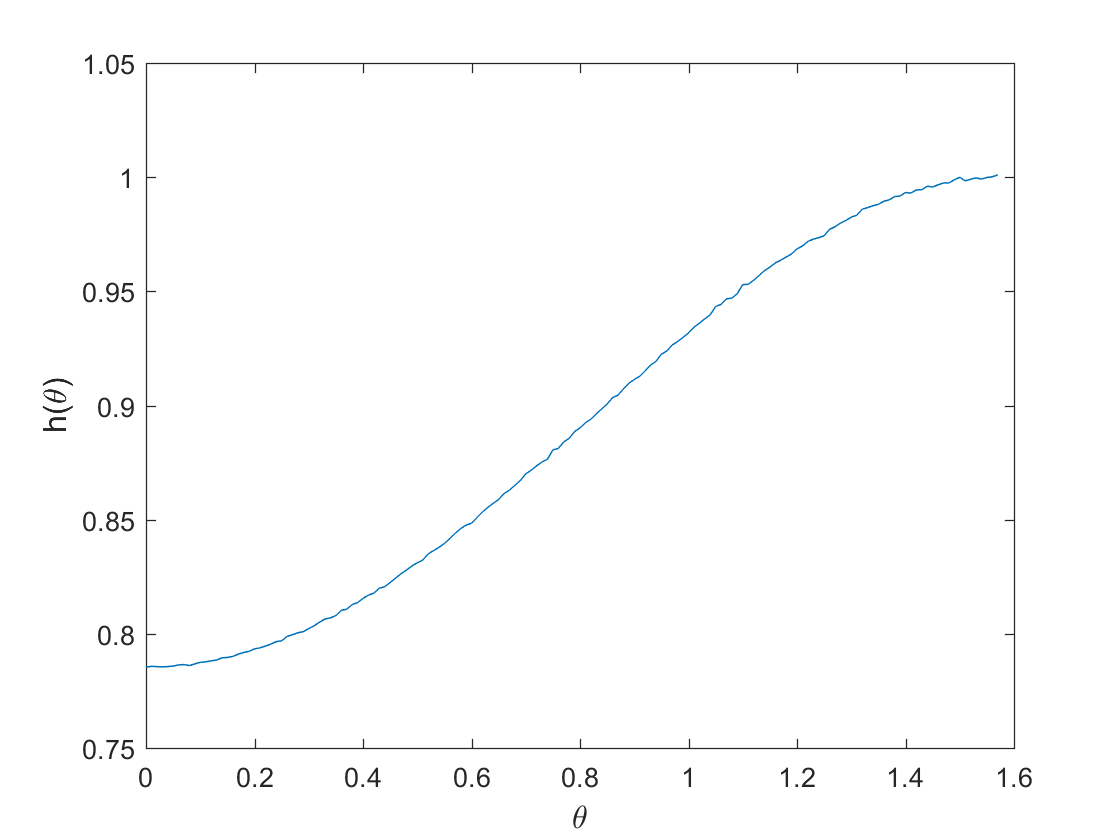}
\includegraphics[width=0.45\textwidth]{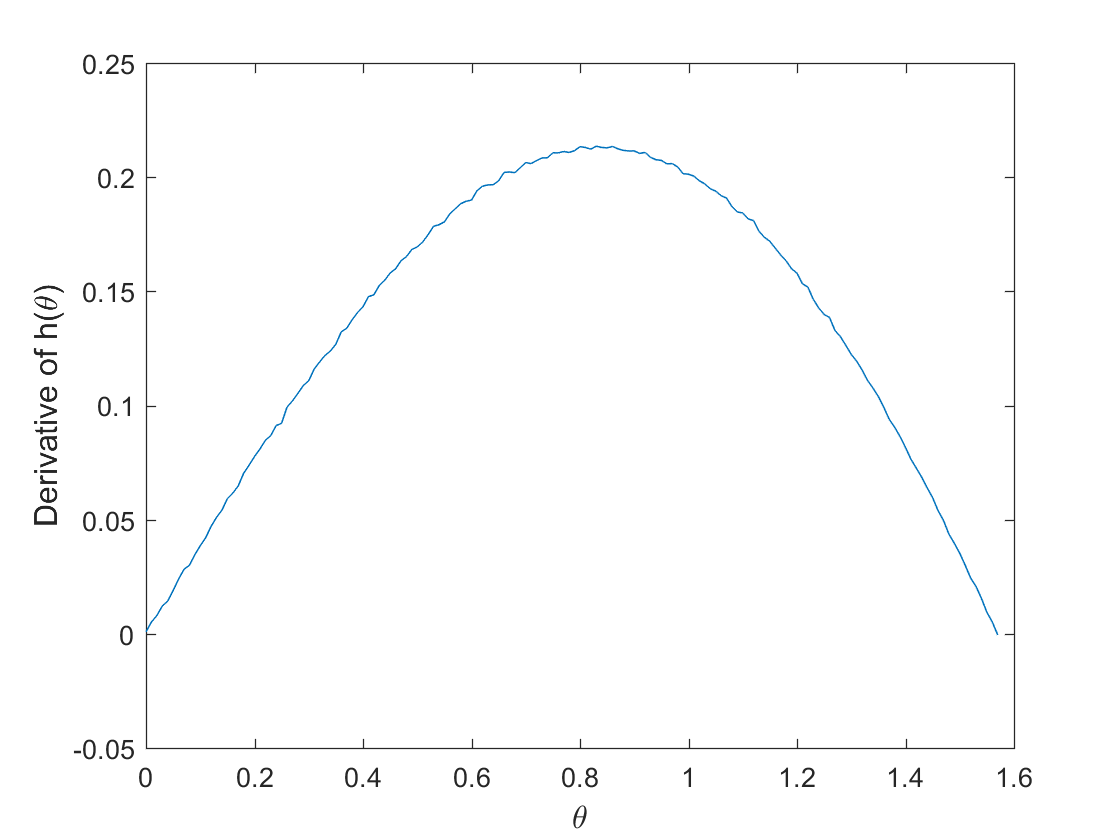}
\end{center}
\caption{$h(\theta)$ and $h'(\theta)$, calculated from the average of $10^6$ simulations.}\label{fig:conjecture}
\end{figure}

\bibliographystyle{abbrv}
\bibliography{bib-online}
\end{document}